\documentclass[a4paper,11pt]{amsart}
 \usepackage{amssymb,amsfonts,amsmath,mathtools}

\newtheorem{Theorem}{Theorem}[section]
\newtheorem{Proposition}[Theorem]{Proposition}
\newtheorem{Lemma}[Theorem]{Lemma}
\newtheorem{Corollary}[Theorem]{Corollary}
\theoremstyle{definition}
\newtheorem{Definition}[Theorem]{Definition}

\newtheorem{Remark}[Theorem]{Remark}
\newtheorem{Notation}{Notation}

\begin{document}

\title[Magnetic Geodesics on the Space of K\"{a}hler Potentials]{Magnetic Geodesics on the Space of K\"{a}hler Potentials}

\author{S\.{i}bel \c{S}ah\.{i}n}

 \date{\today}

\address{Department of Mathematics, Mimar Sinan Fine Arts University, Istanbul, Turkey}

\email{sibel.sahin@msgsu.edu.tr}

\begin{abstract}
In this work, magnetic geodesics over the space of K\"{a}hler potentials are studied through a variational method for a generalized Landau-Hall functional. The magnetic geodesic equation is calculated in this setting and its relation to a perturbed complex Monge-Amp\`{e}re equation is given. Lastly, the magnetic geodesic equation is considered over the special case of toric K\"{a}hler potentials over toric K\"{a}hler manifolds.
\end{abstract}

\maketitle

\tableofcontents


\section*{Introduction}
\label{sec:intro}

Let $(X,\omega)$ be a compact K\"{a}hler manifold. X.X.Chen et al. examined the metric and geometric aspects of the space of all K\"{a}hler metrics $\mathcal{H}_\alpha$ and showed a remarkable result that $\mathcal{H}_\alpha$ is a path metric space \cite{C00,CC02,CT08,C09,CS09}. Weak geodesics (a special type of path between two points in $\mathcal{H}_\alpha$ ) on this space play an important role in the variational approach for solving complex Monge-Amp\`{e}re equations (CMAE) and for understanding the application of CMAE to find the K\"{a}hler-Einstein metrics on various varieties. As we will give in detail in the following parts of this study Semmes \cite{S92} showed that the geodesic equation can be reformulated as a homogenous CMAE of one degree higher.

Magnetic curves (or magnetic geodesics) are generalizations of geodesics. A such curve actually describes the trajectory of a particle moving under the effect of a magnetic field. As it is known geodesics are extremals of the energy functional and the geodesic equation can be calculated via Euler-Lagrange equations. In the case of magnetic geodesics one considers the extremals of Landau-Hall functional with the magnetic force (known as Lorentz force) and the magnetic trajectories ate the curves $\gamma$ that satisfy the Lorentz equation
$$
\nabla_{\gamma'}\gamma'=\phi(\gamma')
$$   
with the magnetic field $\phi$. 

In this study we will take this generalization of magnetic geodesics to the setting of K\"{a}hler potentials over compact K\"{a}hler manifolds and examine its relation to a perturbed CMAE. The organization of the paper is as follows:

In Section 1 we give the necessary background about K\"{a}hler potentials, geodesics over the space of K\"{a}hler potentials $\mathcal{H}_\alpha$ and basics about magnetic geodesics on Riemannian manifolds. In Section 2, we introduce magnetic geodesics over $\mathcal{H}_\alpha$ and we will give the main result of this study about the magnetic geodesic equation. As it is known, toric K\"{a}hler manifolds and toric potentials are very special cases however they are quite useful to test your hypothesis about the non toric environment hence in the last part of this paper we examine the toric magnetic geodesics over toric compact K\"{a}hler manifolds $(X,\omega,\mathbb{T})$. 

\section{Preliminaries}

Throughout this work we will work on the magnetic geodesics on the space of K\"{a}hler potentials so let us first introduce the setting and the classical geodesics of these potentials.

\begin{Definition}
Let $(X,\omega)$ be a compact K\"{a}hler manifold of dimension $n$. Any other K\"{a}hler metric on $X$ that is in the same cohomology class as $\omega$ is given by 
$$
\omega_{\varphi}=\omega+dd^c\varphi
$$
where $d=\partial+\overline{\partial}$ and $d^c=\dfrac{1}{2\pi i}(\partial-\overline{\partial})$. Then the space of K\"{a}hler potentials is defined as 
$$
\mathcal{H}=\{\varphi\in C^\infty(X):~~\omega_{\varphi}=\omega+dd^c\varphi>0\}.
$$
\end{Definition}

\begin{Notation}

\begin{enumerate}
\item[(1)] We know that two K\"{a}hler potentials generate the same metric if and only if they differ by a constant hence
$$
\mathcal{H}_\alpha=\mathcal{H}/\mathbb{R}
$$
is the space of K\"{a}hler metrics on $X$ in the cohomology class $\alpha=\{\omega\}\in H^{1,1}(X,\mathbb{R})$.

\item[(2)] Let $V_\alpha=\alpha^n=\int_X \omega^n$ be the volume of the space $X$. Then for any $\varphi\in\mathcal{H}$ we denote the Monge-Amp\`{e}re measure associated to $\varphi$ as 
$$
MA(\varphi)=\dfrac{\omega_{\varphi}^{n}}{V_\alpha}.
$$ 
\end{enumerate}

\end{Notation}

\subsection*{Part 1: Classical geodesics on the space of K\"{a}hler potentials}

\begin{Definition}
Geodesics between two points $\varphi_0,\varphi_1$ in $\mathcal{H}$ are defined as the extremals of the energy functional 
$$
\varphi\longrightarrow H(\varphi)=\dfrac{1}{2}\int_{0}^{1}\int_{X}(\dot\varphi_t)^2MA(\varphi_t)dt
$$
where $\varphi=\varphi_t$ is a path in $\mathcal{H}$ joining $\varphi_0$ and $\varphi_1$.
\end{Definition}

As it is done in most of the variational problems, the geodesic equation i.e the equation whose solution gives a geodesic path for certain boundary conditions is obtained by computing the Euler-Lagrange equation for this energy functional with fixed end points:

\begin{Lemma}[\cite{G14}, Lemma 1.2]
The geodesic equation is 
\begin{equation}\label{eq:geodesic equation}
\ddot{\varphi}=\Vert \nabla\dot{\varphi}\Vert_{\varphi}^{2}
\end{equation}
where the gradient is relative to the metric $\omega_\varphi$. This identity is also written as
$$
\ddot{\varphi}MA(\varphi)=\dfrac{n}{V_\alpha}d\dot{\varphi}\wedge d^c\dot{\varphi}\wedge \omega_{\varphi}^{n-1}.
$$
\end{Lemma}

\underline{\textit{Boundary Problem for Geodesic Equation:}}

Given $\varphi_0,\varphi_1$ two distinct potentials in $\mathcal{H}$ can one find a path $\varphi=(\varphi_t)_{0\leq t\leq 1}\in \mathcal{H}$ which is a solution of the geodesic equation (\ref{eq:geodesic equation}) with endpoints $\varphi(0)=\varphi_0$ and $\varphi(1)=\varphi_1$ ?

The answer to this problem is given by Semmes \cite{S92} and the solution is totally determined by the solvability of the complex Monge-Amp\`{e}re equation over a certain region. Before giving Semmes' solution let us introduce the specific setting of the problem:

For each path $(\varphi_t)_{0\leq t\leq 1}\in \mathcal{H}$ we set $\varphi(x,t,s)=\varphi_t(x)$, $x\in X$, $e^{t+is}\in A=[0,1]\times \partial\mathbb{D}$. Set $z=e^{t+is}$ and $\omega(x,z)\coloneqq\omega(x)$.

\begin{Proposition}[\cite{S92}] 
The path $\varphi_t$ is a geodesic in $\mathcal{H}$ if and only if the associated radial function $\varphi$ on $X\times A$ is a solution of the homogeneous complex Monge-Amp\`{e}re equation
$$
(\omega+dd^c_{x,z}\varphi)^{n+1}=0
$$
where the derivatives are taken in all variables $x,z$.
\end{Proposition}

\subsection*{Part 2: Magnetic Geodesics on Riemannian Manifolds }

As we have seen before, geodesics are obtained as the critical points of the energy functional. Finding the critical points of a specific perturbation of the energy functional however results in another type of curves namely magnetic geodesics:

Let us follow the definition of \cite{IM14},

\begin{Definition}
Let $(M,g)$ be a Riemannian manifold and $\omega$ be a 1-form (potential). For a smooth curve $\gamma:[a,b]\rightarrow M$ consider the functional
$$
LH(\gamma)\coloneqq\int_{a}^{b} \dfrac{1}{2}\left(\langle \gamma '(t),\gamma '(t)\rangle+\omega(\gamma '(t))\right)dt
$$ 
which is called the Landau-Hall functional for the curve $\gamma$.

The critical points of the LH-functional satisfy the Lorentz equation which is given as 
$$
\nabla_{\gamma '}\gamma'-\phi(\gamma')=0
$$ 
where $\phi$ is a $(1,1)$-tensor field on $M$ and determined by $g(\phi(X,Y))=d\omega(X,Y)$ for all $X,Y$ tangent to M.

\end{Definition}

\begin{Remark}
For a map $f:(M,g)\rightarrow (N,h)$, the Landau-Hall functional is given as 
$$
LH(f)\coloneqq E(f)+\int_N\omega(df(\xi))d\upsilon_h
$$
where $E(f)$ is the energy functional.

A map is called magnetic if it is a critical point of the Landau-Hall integral above.
\end{Remark}

\section{Magnetic geodesics on the space of K\"{a}hler potentials }

The notion of magnetic geodesics can also be generalized to K\"{a}hler potentials on a compact K\"{a}hler manifold and for the rest of the study we will focus on how this can be done. 

Let $(X,\omega)$ be a compact K\"{a}hler manifold.

\begin{Definition}
Magnetic geodesics between two points $\varphi_0,\varphi_1\in\mathcal{H}$ are defined to be the extremals of the generalized Landau-Hall functional
\begin{equation}\label{eq:glh}
\varphi\longrightarrow LH(\varphi)\coloneqq \dfrac{1}{2}\int_{0}^{1}\int_X(\dot{\varphi_t})^2MA(\varphi_t)dt+\int_{0}^{1}\beta_{\dot{\varphi_t}}(\dot{\varphi_t})dt
\end{equation}
where $\varphi=\varphi_t$ is a path connecting $\varphi_0$ and $\varphi_1$ in $\mathcal{H}$ and $\beta_{\dot{\varphi_t}}$ is the closed 1-form (potential) on $\mathcal{H}$ defined as $\beta_{\dot{\varphi_t}}(\dot{\varphi_t})=\int_X(\dot{\varphi_t})MA(\dot{\varphi_t})$.
\end{Definition}

\begin{Remark}
For the details of the closedness of the 1-form $\beta_\varphi$, see \cite{K12}, pp:245-246. 
\end{Remark}

Now we will give the main result of this study and calculate the magnetic geodesic equation for K\"{a}hler potentials:

\begin{Theorem}\label{thm:mge}
Magnetic geodesic equation on the space of K\"{a}hler potentials is given as
\begin{equation*}
\ddot{\varphi} MA(\varphi)=\dfrac{n}{V_\alpha} d\dot{\varphi}\wedge d^c\dot{\varphi} \wedge\omega_{\varphi}^{n-1}-\dfrac{2n}{V_\alpha} dd^c\ddot{\varphi} \wedge\omega_{\dot{\varphi}}^{n-1}-\dfrac{n(n-1)}{V_\alpha} dd^c \dot{\varphi} \wedge dd^c \ddot{\varphi}\wedge \omega_{\dot{\varphi}}^{n-2}.
\end{equation*}
\end{Theorem}
 
\begin{proof}
We need to calculate the critical points of the generalized Landau-Hall functional therefore we will compute the Euler-Lagrange equation of this functional. Suppose that $\phi_{s,t}$ is a variation of $\varphi_t$ with fixed end points such that
$$
\phi_{0,t}=\varphi_t,~~\phi_{s,0}=\varphi_0,~~\phi_{s,1}=\varphi_1.
$$

Let $\psi_t=\dfrac{\partial \phi}{\partial s}\bigg\rvert_{s=0}$ then $\psi_0\equiv\psi_1\equiv 0$ ($\dagger$) and we have 

$$
\phi_{s,t}=\varphi_t+s\psi_t+o(s)~~~~\text{and}~~~~\dfrac{\partial\phi_{s,t}}{\partial t}=\dot{\varphi_t}+s\dot{\psi_t}+o(s).
$$   
Then,
\begin{equation*}
\begin{split}
LH(\phi_{s,t})= &\dfrac{1}{2}\int_{0}^{1}\int_X(\dot{\phi_{s,t}})^2MA(\phi_{s,t})dt+\int_{0}^{1}\int_X(\dot{\phi_{s,t}})MA(\dot{\phi_{s,t}})dt
\\
= &\dfrac{1}{2}\int_{0}^{1}\int_X(\dot{\varphi_t})^2MA(\varphi_t)+\dfrac{ns}{2V_\alpha}\int_{0}^{1}\int_X(\dot{\varphi_t})^2dd^c\psi_t\wedge\omega_{\varphi_t}^{n-1}dt
\\
 + & s\int_{0}^{1}\int_X\dot{\varphi_t}\dot{\psi_t}MA(\varphi_t)dt+\int_{0}^{1}\int_X(\dot{\varphi_t})MA(\dot{\varphi_t})dt
\\
+ & \dfrac{ns}{V_\alpha}\int_{0}^{1}\int_X(\dot{\varphi_t})dd^c\dot{\psi_t}\wedge\omega_{\dot{\varphi_t}}^{n-1}dt+s\int_{0}^{1}\int_X\dot{\psi_t}MA(\dot{\varphi_t})dt+o(s).
\end{split}
\end{equation*}

Now using the boundary values ($\dagger$) let us calculate the s-dependent terms,
\begin{itemize}
\item $\displaystyle{\int_{0}^{1}\int_X(\dot{\varphi_t})^2dd^c\psi_t\wedge\omega_{\varphi_t}^{n-1}dt=2\int_{0}^{1}\int_X\psi_t\{d\dot{\varphi_t}\wedge d^c\dot{\varphi_t}+\dot{\varphi_t}\wedge dd^c\dot{\varphi_t}\}\wedge \omega_{\varphi_t}^{n-1}dt}$

\vspace{0.2cm}

\item $\displaystyle{\int_{0}^{1}\int_X\dot{\varphi_t}\dot{\psi_t}MA(\varphi_t)dt=-\int_{0}^{1}\int_X}\psi_t\{\ddot{\varphi_t}MA(\varphi_t)+\dfrac{n}{V_\alpha}\dot{\varphi_t}dd^c \dot{\varphi_t}\wedge \omega_{\varphi_t}^{n-1}\}dt$

\vspace{0.2cm}

\item $\displaystyle{\int_{0}^{1}\int_X(\dot{\varphi_t})dd^c\dot{\psi_t}\wedge\omega_{\dot{\varphi_t}}^{n-1}dt=\int_{0}^{1}\int_X\dot{\psi_t}dd^c\dot{\varphi_t}\wedge\omega_{\dot{\varphi_t}}^{n-1}}$\\
$\displaystyle{-\int_{0}^{1}\int_X\psi_t\{dd^c\ddot{\varphi_t}\wedge\omega_{\dot{\varphi_t}}^{n-1}+(n-1)dd^c\dot{\varphi_t}\wedge dd^c\ddot{\varphi_t}\wedge\omega_{\dot{\varphi_t}}^{n-2}\}dt}$

\vspace{0.2cm}

\item $\displaystyle{\int_{0}^{1}\int_X\dot{\psi_t}MA(\dot{\varphi_t})dt=-\int_{0}^{1}\int_X\psi_t\left(\dfrac{n}{V_\alpha}dd^c \ddot{\varphi_t}\wedge\omega_{\dot{\varphi_t}}^{n-1}\right)dt}$
\end{itemize}

If we combine all these equations, we obtain
\begin{equation*}
LH(\phi_{s,t})=LH(\varphi_t)+
\end{equation*}

\begin{equation*}
    \begin{split}
    s\int_{0}^{1}\int_X \psi_t \left\{ -\ddot{\varphi_t}MA(\varphi_t)+\dfrac{n}{V_\alpha}d\dot{\varphi_t}\wedge d^c\dot{\varphi_t}\wedge \omega_{\varphi_t}^{n-1}-\dfrac{2n}{V_\alpha} dd^c\ddot{\varphi_t}\wedge \omega_{\dot{\varphi_t}}^{n-1}  \right. \\
    \left. -\dfrac{n(n-1)}{V_\alpha} dd^c\dot{\varphi_t}\wedge dd^c\ddot{\varphi_t}\wedge \omega_{\dot{\varphi_t}}^{n-2} \vphantom{\int_1^2} \right\} dt+o(s). 
    \end{split}
\end{equation*}

Hence, if $\varphi$ is a critical point of the generalized Landau-Hall functional then it satisfies the following equation

\begin{equation*}
\ddot{\varphi} MA(\varphi)=\dfrac{n}{V_\alpha}\left[d\dot{\varphi}\wedge d^c\dot{\varphi} \wedge\omega_{\varphi}^{n-1}- 2dd^c\ddot{\varphi} \wedge\omega_{\dot{\varphi}}^{n-1}-(n-1) dd^c \dot{\varphi} \wedge dd^c \ddot{\varphi}\wedge \omega_{\dot{\varphi}}^{n-2}\right].
\end{equation*}

\end{proof}

\section{A Special Case: Toric Magnetic Geodesics }

In this section we will consider the magnetic geodesic equation on a special setting where $(X,\omega,\mathbb{T})$ is a toric, compact, K\"{a}hler manifold and our potentials are toric K\"{a}hler potentials. Before passing to magnetic geodesic equation let us first give the details about toric manifolds and the structure of toric potentials:

\begin{Definition}
A toric, compact, K\"{a}hler manifold is obtained by an equivariant compactification of the torus $\mathbb{T}=(\mathbb{C}^*)^n$ with an $(S^1)^n$-invariant K\"{a}hler metric $\omega$ and given by the triple $(X,\omega,\mathbb{T})$. 

In this setting the K\"{a}hler metric $\omega$ can be written in the form 
$$
\omega=dd^c F_0\circ L
$$
over $\mathbb{T}$ where $F_0:\mathbb{R}^n\rightarrow\mathbb{R}$ is a smooth, strictly convex function and $L:\mathbb{T}\rightarrow\mathbb{R}^n$ is the logarithmic transformation function
$$
L(z_1,\dots, z_n)=(\log\vert z_1\vert,\dots,\log\vert z_n\vert).
$$

\end{Definition} 

As it is very well known, in the compact setting as a result of the Maximum Modulus Principle there is no non-constant plurisubharmonic function however we have a general class of functions analogous to plurisubharmonic functions namely quasiplurisubharmonic functions in this setting:

\begin{Definition}
A function $\varphi:X\rightarrow R\cup\{-\infty\}$ is called $\omega$-plurisubharmonic if
\begin{itemize}
\item[(i)] it is locally the sum of a plurisubharmonic function and a smooth function,
\item[(ii)] the current $\omega+dd^c \varphi$ is positive on $X$.
\end{itemize}

An $\omega$-plurisubharmonic function is called \textit{toric} if it is invariant under the $(S^1)^n$-action induced by the $(\mathbb{C}^*)^n$ action on $X$. Toric $\omega$-plurisubharmonic functions on $X$ are denoted as $PSH_{tor}(X,\omega)$.
\end{Definition} 

From the definition we obtain a representation of $PSH_{tor}(X,\omega)$ functions over $\mathbb{T}$ such that there exists a convex function $F_\varphi:\mathbb{R}^n\rightarrow \mathbb{R}$ such that 
\begin{equation}\label{eq:torrep}
F_\varphi\circ L=F_0\circ L+\varphi~~\text{on}~~\mathbb{T}\subset X
\end{equation}

The representation (\ref{eq:torrep}) gives a relation between toric $\omega$-plurisubharmonic functions and real convex functions, now we will continue with a result which takes this relation one step further i.e the connection between complex and real Monge-Amp\`{e}re measures [\cite{CGSZ19}, Lemma 2.3]:

\begin{Proposition}
Let $F:\mathbb{R}^n\rightarrow\mathbb{R}$ be a convex function. If $\chi$ is a continuous function with compact support on $\mathbb{R}^n$ then
$$
\int_{\mathbb{T}}(\chi\circ L)(dd^cF\circ L)^n=\int_{\mathbb{R}^n} \chi MA_{\mathbb{R}}(F)
$$
where $MA_{\mathbb{R}}(F)$ is the real Monge-Amp\`{e}re measure of $F$ defined as
$$
MA_{\mathbb{R}}(F)=n!\det\left[\dfrac{\partial^2F}{\partial x_i\partial x_j}\right]dV.
$$ 
\end{Proposition}
\begin{Remark}
For a detailed study of the concepts related to toric pluripotential theory, see \cite{CGSZ19}. 
\end{Remark}

Now let us introduce the space of toric-K\"{a}hler potentials:

\begin{Definition}
The space of toric-K\"{a}hler potentials are defined as
$$
\mathcal{H}_{tor}=\mathcal{H}\cap PSH(X,\omega)
$$
where a toric-K\"{a}hler potential is represented by a strictly convex function.
\end{Definition}

Lastly, let us give the magnetic geodesic equation for toric-K\"{a}hler potentials in totally real terms:

\begin{Corollary}
In the space of toric-K\"{a}hler potentials $\mathcal{H}_{tor}(X,\omega)$ the magnetic geodesic equation can be written in the following form:
\begin{equation}
\begin{split}
\ddot{F_\varphi}MA_{\mathbb{R}}(F_\varphi) & =\dfrac{n}{V}\left(\sum_{i,j}\left(\dfrac{\partial\dot{F_\varphi}}{\partial x_i}\right)\left(\dfrac{\partial\dot{F_\varphi}}{\partial x_j}\right)\right)\left(\sum_{i,j}\dfrac{\partial^2 F_\varphi}{\partial x_i\partial x_j}\right)^{n-1}\\
&-\dfrac{2n}{V}\left(\sum_{i,j}\dfrac{\partial^2 \ddot{F_\varphi}}{\partial x_i\partial x_j}\right)\left(\sum_{i,j}\dfrac{\partial^2 \dot{F_\varphi}}{\partial x_i\partial x_j}\right)^{n-1}\\
& -\dfrac{n(n-1)}{V}\left(\sum_{i,j}\dfrac{\partial^2 \dot{F_\varphi}}{\partial x_i\partial x_j}\right)\left(\sum_{i,j}\dfrac{\partial^2 \ddot{F_\varphi}}{\partial x_i\partial x_j}\right)\left(\sum_{i,j}\dfrac{\partial^2\dot{F_\varphi}}{\partial x_i\partial x_j}\right)^{n-2}
\end{split}
\end{equation}

where $F_\varphi$ is the corresponding path of strictly convex functions for the path $\varphi$ connecting $\varphi_0,\varphi_1\in \mathcal{H}_{tor}$ and $\displaystyle{V=\int_{\mathbb{R}^n}MA_{\mathbb{R}}(F_0)}$.

\end{Corollary}

\begin{proof}
Representation of toric functions together with the previous proposition and the main theorem, (Theorem \ref{thm:mge}) give the result.
\end{proof}

\end{document}